\documentclass[11pt,a4paper]{amsart}
\usepackage{amsfonts,amssymb,graphics,verbatim,bbm,setspace, amsmath}
\usepackage[usenames,dvipsnames]{color}
\usepackage[left=2cm, right=2cm, top=2.8cm, bottom=2.5cm]{geometry}
\usepackage{enumitem}
\usepackage{ifpdf}

\ifpdf
\usepackage[pdftex,colorlinks,bookmarks]{hyperref}
\hypersetup{%
pdftitle={Speed of convergence to the QSD for L\'evy input fluid queues},
pdfauthor={Palmowski, Vlasiou},
pdfkeywords={quasi, stationary, distribution, levy, zbyszek, zbigniew, fluid, queue},
bookmarks=true,%
bookmarksnumbered=true,
pdfstartview={FitH},
colorlinks=true,
citecolor=Plum,
linkcolor=RedOrange,
urlcolor=RawSienna,
bookmarksopen=false,%
bookmarksopenlevel=1,
unicode=true,%
breaklinks=true,%
}%
\else
  \newcommand\phantomsection\relax
  \newcommand{\url}[1]{#1}
  \newcommand{\href}[2]{#2}
\fi

\newtheorem{theorem}{Theorem}

\newtheorem{proposition}[theorem]{Proposition}
\newtheorem{definition}[theorem]{Definition}

\newtheorem{corollary}[theorem]{Corollary}
\theoremstyle{definition}
\newtheorem{remark}[theorem]{Remark}
\newtheorem{example}{Example}

\newcommand{\imi}{{\rm i}}

\newcommand{\EE}{{\mathbb E}}
\newcommand{\PP}{{\mathbb P}}
\newcommand{\s}{^*}

\newcommand{\halmos}{\vspace{3mm} \hfill $\Box$}

\usepackage{soul}

\makeatletter
\renewcommand{\paragraph}{\@startsection{paragraph}{4}  {\z@}{0.5\baselineskip \@plus 1ex \@minus .2ex}{-1em}{\normalfont\normalsize\it\bfseries}}
\makeatother

\begin{document}
\title{Speed of convergence to the quasi-stationary\\ distribution for L\'evy input fluid queues}
\author[Z.\ Palmowski and M.\ Vlasiou]{Zbigniew Palmowski}
\address{Faculty of Pure and Applied Mathematics\\
Wroc\l aw University of Science and Technology\\
Wyb. Wyspia\'nskiego 27, 50-370 Wroc\l aw, Poland}
\email{zbigniew.palmowski@gmail.com}
\author[]{Maria Vlasiou}
\address{Department of Mathematics and Computer Science\\ Eindhoven University of Technology\\ The Netherlands}
\email{m.vlasiou@tue.nl}

\thanks{This work is partially supported by an NWO individual grant through project 632.003.002
and by the Ministry of Science and Higher Education of Poland under the grant 2013/09/B/ST1/01778
(2013-2016).}

\date{\today}

\begin{abstract}
In this note we prove that the speed of convergence of the workload of a L\'evy-driven queue to the quasi-stationary distribution is of order $1/t$.
We identify also the Laplace transform of the measure giving this speed and provide some examples.
\vspace{3mm}

\end{abstract}
\maketitle

\section{Introduction}
In this paper, we consider a storage system with L\'evy netput. In other words,  the workload process $\{Q(t), t\geq 0\}$
is a spectrally one-sided L\'{e}vy process $X(t)$ that is reflected at $0$:
\begin{equation}\label{workload}
Q(t):=x+X(t)-\inf_{s\leq t} (x+X(s))^-.
\end{equation}
We assume that the drift of the process $X(t)$ is negative; that is, we have $\EE X(1)<0$.
This stability condition guarantees the existence of a stationary distribution $\pi$ of $Q$,
which by virtue of ‘Reich’s identity' can be expressed in term of the all-time supremum:
\begin{equation}\label{pi}
\pi(x)=\PP\left(\sup_{t\geq 0} X(t)\leq x\right).
\end{equation}
In the sequel, we consider the initial distribution $Q(0)$ sampled from this steady-state distribution which
is indicated by adding the subscript $\pi$ to the probability measure $\PP$ and to the associated expectation $\EE$.

Now, let $T$ denote the {\it busy period}; that is $$T=\inf\{t\ge0:\ Q(t)=0\}.$$
We will further consider the Yaglom limit
\[
\label{def2}
\lim_{t\rightarrow\infty}\PP_\pi(Q(0)\in d x, Q(t)\in d y\,|\,T >t):= \mu (d x, d y),
\]
where the convergence is to be understood in the weak sense.
Yaglom limits are a probability measure and a particular case of the {\it quasi-stationary (QS) distribution},
which is and invariant distribution for the process conditioned on non-extinction; that is, we condition on the event that the process survives some killing event (e.g.\ related with
exiting from some subset of possible values).

Yaglom \cite{MR0022045} was the first to explicitly  identify QS
distributions for the subcritical Bienaym\'e-Galton-Watson branching process.
This result has been generalized in the context of
the continuous-time branching process and the Fleming-Viot process; see \cite{MR3498004, MR2318407, MR2299923}.
Similar results were also derived for Markov chains on positive integers with an absorbing state at the origin; see Seneta and Vere-Jones \cite{Seneta}, Tweedie
\cite{tweedie}, Jacka and Roberts \cite{jackaroberts} and the bibliographic database of Pollet \cite{pollett}. Recently, Foley and McDonald showed that Yaglom limits may depend on the starting state \cite{foley17}.

Work on QS distributions has been very extensive.
Martinez and San Martin \cite{Martmart} analyze the Brownian motion with drift exiting from the positive half-line,
complementing the result for random walks obtained by Iglehart \cite{Iglehart}.
Later, QS laws have been studied for various L\'evy processes.
Kyprianou \cite{kyprianou} found the Laplace transform of the QS distribution for the workload process of the stable $M/G/1$
queue with service times that have a rational moment generating function.
Kyprianou and Palmowski \cite{kyprpalm} identified the quasi-stationary
distribution associated with a general light-tailed L\'{e}vy process. Haas and Rivero \cite{rivero} found (after appropriate scaling)
the QS distribution  when the L\'{e}vy process under study has a jump measure with a regularly varying tail. The speed of convergence (in total variation) to the quasi-stationary distribution for population processes has been studied in \cite{champagnat16}. Finally, Mandjes et al.\ \cite{MPR} derived the QS distribution of the workload process $Q(t)$. This paper builds upon \cite{MPR}.

A contribution of this paper lies in proving that the speed of convergence to the quasi-stationary distribution is surprisingly slow (of order $1/t$). We also identify a measure $\xi(dx,dy)$ (which we call {\it second-order quasi-stationary measure}), such that
\begin{equation}\label{speeddef}
\lim_{t\rightarrow\infty}t\times |\PP_\pi(Q(0)\in d x, Q(t)\in d y\,|\,T >t)-\mu (d x, d y)|=\xi(dx,dy).
\end{equation}
We hence prove the conjecture posed in Polak and Rolski \cite{PR}, which proved the above statement for a birth-death process by using an asymptotic expansion of a transition function
and certain properties of Bessel functions. In this paper, we suggest new method, which relies on a refined Tauberian-type expansion of the Laplace transform. We also analyze in detail the $M/M/1$ queue and a Brownian-driven queue.

If we want to simulate the quasi-stationary distribution directly from definition
\eqref{def2}, then the result stated in \eqref{speeddef} shows that the speed of such a simulation is
very slow. Still, in the literature there are papers giving other efficient algorithms of simulation of quasi-stationary measures; see e.g.\
Blanchet et al.\ \cite{Jose} and references therein.

The main result in \eqref{speeddef} contrasts the typical results derived for the regular stationary distribution of Markov processes where, in most of the cases, the rate is exponential.
More precisely, for many models the distance between the distribution of the stochastic process at time $t$ and its stationary
distribution decays exponentially fast in $t$. The typical distances used are the total variation distance, the separation distance, and the $L^2$ distance.
The classical results concern mainly Markov chains
and use Perron and Frobenius theory, renewal equations or the coupling method; see e.g.\ \cite{2,3,4,5,6}
and references therein.
Another method concerns Harris recurrent Markov processes and
it is based on the construction of a
special Lyapunov function and  then the application of Foster-Lyapunov criteria; see e.g.\ \cite{1, 7, 8}.
All the above-mentioned methods though are different from the one used in this paper, which is based on expansions of Laplace transforms.

The paper is organized as follows.
In next section, we introduce the notation and basics facts that are used later.
In Section \ref{qsd}, we present the main results.
The central step for the proof of the main results is given in Section \ref{mainstep}.
Finally, the last section provides some examples.

\section{Preliminaries}\label{prel}
We follow \cite{kyprianoubook} for definitions,  notations and basic facts on L\'evy processes. Let $X\equiv (X(t))_t$ be a {\it spectrally negative}
L\'{e}vy process, which is defined on the filtered space $(\Omega,\mathcal{ F},\{\mathcal{F}_t\}_{t\geq 0}, \PP)$ with the natural filtration that satisfies the usual assumptions of right continuity and completion. We define $\PP_x$ as $\PP_x(X(0)=x)=1$ and $\PP_0=\PP$; similarly, $\EE_x$ is the expectation with respect to $\PP_x$. We denote by $\Pi(\cdot)$ the jump measure of $X$, which is supported by spectral negativity in the non-positive half-line; in other words, jumps are non-positive.
We define the Laplace exponent $\psi({\eta})$ by
\begin{equation}\label{psi}
\EE e^{{\eta} X(t)}=e^{t\psi({\eta})},
\end{equation}
for ${\eta}\in \mathbb{R}$ such that the left hand side of \eqref{psi} is well-defined (which holds at least for $\eta \geq 0$).
We denote by $\Phi(s):=\inf\{{\eta}\ge0: \psi({\eta})>s\}$ the right inverse of $\psi$; see \cite{kyprianoubook} for details.

\paragraph{Dual process} We also consider the {\it dual process} $\hat{X}_t=-X_t$ with jump measure $\hat{\Pi}\left(0,y\right)=\Pi\left(-y,0\right)$.
Note that $\hat{X}(t)$ is a spectrally positive process having only non-negative jumps.
Characteristics of $\hat{X}$ are indicated by using the same symbols as for $X$, but with a `$\hat{\hspace{3mm}}$' added.
In particular,
\begin{equation}\label{psidual}
\hat{\psi}(\eta)= t^{-1} \log \EE e^{{\eta} \hat{X}(t)}=t^{-1} \log \EE e^{-{\eta} X(t)}=\psi(-\eta).
\end{equation}

We skip the symbol `$\hat{\hspace{3mm}}$' for $Q(t)$ and hence for $T$ and all quantities related to $Q$ as it will be clear from the context if a statement concerns the spectrally negative or the spectrally positive case.

\paragraph{Asymptotic expansions}

Consider a function $f:\mathbb{R}\to\mathbb{R}$ such that $f(x)=0$ for $x<0$. Let $\tilde f(z):=\int_0^\infty e^{-z x}f(x)\,d x$ for $z\in \mathbb{C}$ be its Laplace transform. Consider singularities of $\tilde{f}(z)$; among these, let $a_0<0$ be the one  with the largest real part. Notice that this yields the integrability of $\int_0^\infty|f(x)|\, d x$. The inversion formula reads
\[
f(x)=\frac{1}{2\pi \imi}\int_{a-\imi \infty}^{a+\imi\infty}\tilde{f}(z) e^{zx}\, d z
\]
for some (and then any) $a>a_0$. In this paper, we need the following {\it Tauberian theorem}, found in Doetsch \cite[Theorem~37.1]{Doetsch1974}, where the behaviour of the Laplace transform around the singularity $a_0<0$ plays a crucial role.

First recall the concept of the $\mathfrak W$-contour, centered at $a_0$, with a half-angle of opening $\pi/2<\psi\le \pi$, as depicted on \cite[Fig.\ 30, p.\ 240]{Doetsch1974} ; also, for the purposes of our problem, ${\mathcal G}_{a_0}(\psi)$  is the region between the contour $\mathfrak W$ and the line $\Re(z)=0$. More precisely,
\[
{\mathcal G}_{a_0}(\psi) \equiv \{z \in \mathbb{C}; \Re(z) < 0, z \ne a_0, |\arg (z - a_0)| < \psi \},
\]
where $\arg z$ is the principal part of the argument of the complex number $z$. In the following theorem, conditions  are identified that provide an asymptotic expansion of the Laplace transform.

\begin{theorem}[\protect{\cite[Theorem~37.1]{Doetsch1974} }]\label{doetsch}
Suppose that for  $\tilde f:\mathbb{C}\to \mathbb{C}$ and $a_0 < 0$ the following three conditions hold:
\begin{enumerate}[label=(A\arabic*)]
\item \label{cond:A1} $\tilde f(\cdot)$ is analytic in a region ${\mathcal G}_{a_0}(\psi)$ for some $\pi/2<\psi\le \pi$;
\item $\tilde f(z) \to 0$ as $|z| \to \infty$ for $z \in {\mathcal G}_{a_0}(\psi)$;
\item for some constants $c_\nu$, $\tilde f(s)$ has in $|\mathrm{arc} (s - a_0)| < \psi$ the asymptotic expansion
\begin{equation}\label{eqn:bridge case 1}
\tilde f(s)\approx\sum_{\nu=0}^\infty c_\nu (s-a_0)^{\lambda_\nu}, \qquad (\Re(\lambda_0)<\Re(\lambda_1)<\ldots)\text{\ as\ }s\to a_0.
\end{equation}
\end{enumerate}
Then we conclude that as $t\to\infty$, $f(t)$ has the asymptotic expansion
\[
f(t)\approx e^{a_0 t}\sum_{\nu=0}^\infty\frac{c_\nu}{\Gamma(-\lambda_\nu)}\frac{1}{t^{\lambda_\nu+1}},\qquad \left(\frac{1}{\Gamma(-\lambda_\nu)}=0\text{\ for\ } \lambda_\nu=0,1,2,\ldots  \right).
\]
\end{theorem}

\paragraph{Assumptions}
For a spectrally negative L\'evy process $X(t)$, we impose the following assumptions:
\newline {\bf (SN)} \:\: There exists ${{\vartheta}_-}>0$ such that:
\begin{enumerate}[label=(SN\arabic*)]
\item $\psi({\vartheta})< \infty$  for $0<{\vartheta}<{{\vartheta}_-}$;
\item    $\psi({\vartheta})$ attains its strictly negative minimum at ${\vartheta}{\s}>0$, where
$0<{\vartheta}{\s}<{{\vartheta}_-}$ (and hence
$\psi'({\vartheta}{\s})=0$);
\item $\Phi$ is analytic in
$\mathcal{G}_{\zeta{\s}}(\phi)$ for $ \pi/2< \phi\leq\pi$, where
\begin{equation}\label{zeta1}
\zeta{\s}:=\psi({\vartheta}{\s})<0.
\end{equation}
\end{enumerate}

Similar conditions are assumed for a spectrally positive L\'evy process $\hat{X}(t)=-X(t)$:
\newline {\bf (SP)} \:\: There exists ${{\vartheta}_+}<0$ such that
\begin{enumerate}[label=(SP\arabic*)]
\item $\hat\psi({\vartheta})< \infty$  for ${{\vartheta}_+}<{\vartheta}$;
\item    $\hat{\psi}({\vartheta})$ attains its strictly negative minimum at ${\vartheta}{\s}<0$, where
${{\vartheta}_+}<{\vartheta}{\s}<0$ (and hence
$\hat{\psi}'({\vartheta}{\s})=0$);
\item $\hat{\Phi}$ is analytic in
$\mathcal{G}_{\zeta{\s}}(\phi)$ for $ \pi/2< \phi\leq\pi$, where
\begin{equation}\label{zeta2}
\zeta{\s}:=\hat{\psi}({\vartheta}{\s})<0.
\end{equation}
\end{enumerate}

To check the above assumptions, we can use the concept of semiexponentiality  of a function $f$  (see \cite[p.\ 314]{henrici}).
\begin{definition}[Semiexponentiality]
A function $f$  is said to be {\em semiexponential} if for some  $0< \phi\le\pi/2$, there exists a finite and strictly negative function $\gamma({\vartheta})$, called the \emph{indicator function}, defined  as the infimum of all $a\in\mathbb{R}$ such that
\[\left|f(e^{\imi{\vartheta}}r)\right|<e^{ar}\]
for all sufficiently large $r$;
here $-\phi\le {\vartheta}\le \phi$
and $\sup \gamma({\vartheta})<0$.
\end{definition}
It was proved in \cite{MPR} using \cite[Thm. 10.9f]{henrici} that if there exists a density of $\Pi$ (resp.\ $\hat{\Pi}$) which is of semiexponential type,
then $\Phi$ (resp.\ $\hat{\Phi}$) is analytic in
$\mathcal{G}_{\zeta{\s}}(\phi)$ for $ \pi/2< \phi\leq\pi$.
In particular, this assumption holds for example
for a linear Brownian motion $X(t)=\sigma B(t)-c t$, where $c>0$ and $B$ is the standard Brownian motion.

\paragraph{Quasi-stationary distribution}
We define a new probability measure  $\PP_x^{{\eta}}$, needed in the sequel, by the exponential change of measure
\[\left. \frac{d \PP_{x}^{{\eta}}}{d \PP_{x}}\right|_{\mathcal{F}_{t}}=e^{{\eta}\big(X(t)-x\big)-\psi({\eta})t}.\]
The quasi-stationary distribution of the workload process $Q(t)$ in the stationary regime was identified in \cite{MPR}.
\begin{theorem}[Mandjes et.\ al \cite{MPR}]\label{qsmain}
\item[(i)]
If $X$ is a spectrally negative L\'evy process satisfying conditions {\bf (SN)}, then
\begin{eqnarray*}
\mu(dx, dy)=Q_-({\vartheta}{\s})^2ye^{-{\vartheta}{\s} (x+y)}
e^{-\Phi(0)x}\hat{V}_{{\vartheta}{\s}}(x)\,dx\,dy\; {\boldsymbol
1}_{\{x\geq 0,y\geq 0\}},
\end{eqnarray*}
where
$Q_-:=(\int_0^\infty
e^{-(\Phi(0)+{\vartheta}{\s})z}V_{{\vartheta}{\s}}(z)\,d
z)^{-1}$
and $\hat{V}_{{\vartheta}{\s}}(x)$ is a renewal function related with the downward ladder height process of $X$
considered on the shifted $\PP_x^{{\vartheta}{\s}}$ probability measure.
\item[(ii)]
If $X$ is a spectrally positive L\'evy process satisfying conditions {\bf (SP)}, then
\[\mu(d x,d y)=Q_+
\left(-\hat\psi'(0+)\frac{\hat\psi({\vartheta}{\s})-{\vartheta}{\s}\hat\psi'({\vartheta}{\s})}{(\hat\psi({\vartheta}{\s}))^2}\right)
 e^{{\vartheta}{\s} (y-x)}x\,V_{-{\vartheta}{\s}}(y)\,\pi(dx) dy\,\;
{\boldsymbol 1}_{\{x\ge 0,y\ge 0\}},\] where $Q_+:=(\int_0^\infty
e^{{\vartheta}{\s} z}V_{-{\vartheta}{\s}}(z)\,{\rm
d}z)^{-1}$
and $V_{{-\vartheta}{\s}}(x)$ is a renewal function related with the upward ladder height process of $X$
considered on the shifted $\PP_x^{{-\vartheta}{\s}}$ probability measure.
\end{theorem}

Let $\mu_R(dy):=\mu(\mathbf{R},dy)$. Then integrating the above results over $dx$ replicates
the quasi-stationary distribution for the spectrally one-sided L\'evy process
related to the classical ruin time derived in \cite{kyprpalm}.
There, it is stated that for a general L\'evy process belonging to the so-called classes A or B,
\[\mu_R(dy) = \vartheta^* \kappa_{\vartheta^*} (0,\vartheta^*)e^{- \vartheta^* y}V_{\vartheta^*}(y)dy,\]
where $\psi(\vartheta)$ attains its strictly negative minimum at $\vartheta^*$ and $\kappa_{\vartheta^*} (\alpha,\beta)$
is the Laplace exponent of the bivariate upward ladder height process under $\PP_x^{\vartheta^*}$.
Hence by (\ref{psidual}), we have that for $\hat{X}$ spectrally positive
\[
\mu_R(d y)= Q_+{e^{{\vartheta}{\s} y}V_{-{\vartheta}{\s}}(y)\,d y
{\boldsymbol 1}_{\{y\geq 0\}}}.
\]
In addition, when $X$ is spectrally negative
\[\mu_R (d y)= ({\vartheta}{\s})^2ye^{-{\vartheta}{\s} y}
d y{\boldsymbol 1}_{\{y\geq 0\}}.
\]


Denote by
\[\tilde{\mu}(\alpha, \beta):=\int_0^\infty \int_0^\infty e^{-\alpha x}e^{-\alpha y}\mu(dx,dy)\]
the bivariate Laplace transform of the quasi-stationary measure. From Proposition \ref{qsmain} stated below (see also \cite{MPR}), we obtain the following equivalent expressions.

\begin{corollary}
\item[(i)]
Under {\bf (SN)},
\begin{equation*}\label{LTQSSN}
\tilde{\mu}(\alpha,\beta)=
\frac{-\psi({\vartheta}{\s})}{\psi(\alpha+\Phi(0))-\zeta{\s}}\frac{({\vartheta}{\s})^2}{({\vartheta}{\s}+\beta)^2}\;.
\end{equation*}
\item[(ii)]
Under {\bf (SP)},
\begin{equation*}\label{LTQSSP}
\tilde{\mu}(\alpha,\beta)=
\hat{\psi}^2({\vartheta}{\s})\cdot
\frac{\hat{\psi}(\alpha+{\vartheta}{\s})-(\alpha+{\vartheta}{\s})\hat{\psi}'(\alpha+{\vartheta}{\s})}{\hat{\psi}^2(\alpha+{\vartheta}{\s})(\hat{\psi}({\vartheta}{\s})
-\hat{\psi}(\beta))}\;.
\end{equation*}
\end{corollary}

The key component of the proof of this corollary is based on Wiener-Hopf factorization, from which master formulas can be derived (given below), and on either some expansion theorems (see \cite{BertoinDoney96} and \cite{kyprpalm}) or some Tauberian-type theorems.

\paragraph{Master formulas}
Recall that $Q(t)$ given in (\ref{workload}) is a workload process with stationary distribution (\ref{pi}) and busy period $T$.
We define now the double Laplace-Stieltjes transform:
\[
L({\vartheta};\alpha,\beta):=\int_0^\infty e^{-{\vartheta} t} \EE_\pi[e^{-\alpha Q(0)-\beta Q(t)},T>t] \, d t.
\]
In \cite{MPR}, the following representations of $L$ were derived.
\begin{proposition}[Mandjes et.\ al \cite{MPR}] \label{p.joint.transform}
\item[(i)]
Under {\bf (SN)},
\begin{equation*}
L({\vartheta};\alpha,\beta)= \frac{\Phi({\vartheta})-\alpha-\Phi(0)}{\Phi({\vartheta})+\beta}\frac{\Phi(0)}{\alpha+\beta+\Phi(0)}\frac{1}{{\vartheta}-\psi(\alpha+\Phi(0))}.\label{eq.Lsp.neg}
\end{equation*}
\item[(ii)]
Under {\bf (SP)},
\begin{equation*}
L({\vartheta};\alpha,\beta)=\frac{\hat{\psi}'(0+)}{{\vartheta}-\hat{\psi}(\beta)}\left(\frac{\alpha+\beta}{\hat{\psi}(\alpha+\beta)} - \frac{\alpha+\hat{\Phi}({\vartheta})}{\hat{\psi}(\alpha+\hat{\Phi}({\vartheta}))}\right).\label{eq.Lsp.pos}
\end{equation*}
\end{proposition}
From the proposition above, it follows that under assumptions {\bf (SN)} or {\bf (SP)} one can extend analytically $L({\vartheta};\alpha,\beta)$ into ${\mathcal G}_{\zeta{\s}}(\psi)$ for some $\pi/2<\psi\le
\pi$.

\section{Main results}\label{qsd}
We state now the main results of this paper. Define the constants
\begin{align*}
A_1&:=\sqrt{\frac{2}{ \psi{''}(\vartheta^*)}},&B_1&:=\sqrt{\frac{2}{ \hat{\psi}{''}(\vartheta^*)}},\\
A_2&:=\frac{\psi{'''}(\vartheta^*)}{3\psi{''}(\vartheta^*)^2}, &B_2&:=\frac{\hat{\psi}{'''}(\vartheta^*)}{3\hat{\psi}{''}(\vartheta^*)^2},\\
A_3&:=-\frac{7}{18 \sqrt{2}} \frac{\psi^{(3)}(\vartheta ^*)^2}{\psi{''}(\vartheta ^*)^{7/2}}-\frac{1}{6 \sqrt{2}}\frac{\psi^{(4)}(\vartheta^*)}{\psi{''}(\vartheta^*)^{5/2}}, &B_3&:=-\frac{7}{18 \sqrt{2}} \frac{\hat{\psi}^{(3)}(\vartheta ^*)^2}{\hat{\psi}{''}(\vartheta ^*)^{7/2}}-\frac{1}{6 \sqrt{2}}\frac{\hat{\psi}^{(4)}(\vartheta^*)}{\hat{\psi}{''}(\vartheta^*)^{5/2}}.
\end{align*}
We start with the following expansion for the double Laplace-Stieltjes transform.
\begin{proposition}\label{prop:joint.transform-expansion}
If {\bf (SN)} or {\bf (SP)} hold then
\[
L({\vartheta};\alpha,\beta)=C_0(\alpha,\beta)+ C_1(\alpha,\beta) (\vartheta-\zeta^*)^{1/2} + C_2(\alpha,\beta) ({\vartheta}-\zeta^{*}) + C_3(\alpha,\beta) (\vartheta-\zeta^{*})^{3/2} + o((\vartheta-\zeta^{*})^{3/2})
\]
for $\zeta^*<0$ defined in (\ref{zeta1}) and (\ref{zeta2}).
\begin{enumerate}[label=(\roman*)]
  \item Under {\bf (SN)},
\begin{align*}
C_0(\alpha,\beta)&= -\frac{\Phi (0) }{\alpha +\beta +\Phi (0)}    \frac{\alpha -\vartheta ^*+\Phi (0)}{\left(\beta +\vartheta^*\right)  \left(\zeta ^*-\psi (\alpha  +\Phi(0))\right)},\\
C_1(\alpha,\beta)&= \frac{A_1 \Phi (0)}{\left(\beta +\vartheta ^*\right)^2 \left(\zeta ^*-\psi (\alpha+\Phi (0))\right)},\\
C_2(\alpha,\beta)&= \frac{\Phi (0)\left(\frac{\left(\beta+\vartheta ^*\right)^2 \left(\alpha -\vartheta ^*+\Phi(0)\right)}{\alpha +\beta+\Phi (0)}-\left(A_2 \left(\beta +\vartheta^*\right)-A_1^2\right)
   \left(\psi (\alpha +\Phi(0))-\zeta^*\right)\right)}{\left(\beta+\vartheta ^*\right)^3\left(\zeta ^*-\psi (\alpha+\Phi (0))\right)^2},\\
C_3(\alpha,\beta)&=\frac{\Phi (0) A_3}{\left(\beta+\vartheta ^*\right)^2\left(\zeta ^*-\psi (\alpha+\Phi(0))\right)}
    -\frac{\Phi (0) }{\left(\beta+\vartheta ^*\right)^4\left(\zeta ^*-\psi (\alpha+\Phi(0))\right)^2}\\
   &\qquad\cdot \left(A_1 \left(\beta +\vartheta ^*\right) \left(2 A_2 \left(\zeta^*-\psi (\alpha +\Phi(0))\right)+\beta +\vartheta^*\right)+A_1^3\left(\psi (\alpha +\Phi(0))-\zeta^*\right)\right).
\end{align*}

  \item Under {\bf (SP)},
\begin{align*}
C_0(\alpha,\beta)&=
\left(\frac{\alpha +\beta }{\hat{\psi }(\alpha +\beta
   )}-\frac{\alpha +\vartheta ^*}{\hat{\psi }\left(\alpha +\vartheta
   ^*\right)}\right)\frac{\hat{\psi }'(0)}{\zeta ^*-\hat{\psi}\left(\beta\right)},       \\
 C_1(\alpha,\beta)&=  -\hat{\psi }'(0)B_1\frac{ \hat{\psi }\left(\alpha +\vartheta
   ^*\right)-(\alpha+\vartheta ^*)
   \hat{\psi }'\left(\alpha +\vartheta ^*\right)}{\hat{\psi }\left(\alpha +\vartheta ^*\right)^2 \left(\zeta^*-\hat{\psi }(\beta)\right)},      \\
 C_2(\alpha,\beta)&=\frac{ \hat{\psi }'(0)}{\zeta ^*-\hat{\psi }(\beta)} \Bigg[\frac{\left(\alpha
   +\vartheta ^*\right) \left(B_1^2 \left(\hat{\psi }\left(\alpha +\vartheta
   ^*\right) \hat{\psi }''\left(\alpha +\vartheta ^*\right)-2 \hat{\psi
   }'\left(\alpha +\vartheta ^*\right)^2\right)+2 B_2 \hat{\psi
   }\left(\alpha +\vartheta ^*\right) \hat{\psi }'\left(\alpha +\vartheta
   ^*\right)\right)}{2 \hat{\psi }\left(\alpha +\vartheta
   ^*\right)^3}      \\
 &+\frac{B_1^2 \hat{\psi }'\left(\alpha +\vartheta
   ^*\right)}{\hat{\psi }\left(\alpha +\vartheta ^*\right)^2} -\frac{B_2}{\hat{\psi }\left(\alpha +\vartheta
   ^*\right)} -\frac{\frac{\alpha
   +\beta }{\hat{\psi }(\alpha +\beta )}-\frac{\alpha +\vartheta ^*}{\hat{\psi
   }\left(\alpha +\vartheta ^*\right)}}{\zeta ^*-\hat{\psi }(\beta)}\Bigg],
\end{align*}
\end{enumerate}
\begin{align*}
 C_3(\alpha,\beta)&=\hat{\psi }'(0) \Bigg[B_1\frac{\hat{\psi}(\alpha +\vartheta
   ^*)-\hat{\psi}'(\alpha +\vartheta ^*)(\alpha+ \vartheta ^*)}{\hat{\psi}(\alpha +\vartheta ^*)^2 \left(\zeta^*-\hat{\psi}(\beta )\right)^2}     \\
   &-\frac{1}{6 \hat{\psi }\left(\alpha +\vartheta ^*\right)^4 \left(\zeta
   ^*-\hat{\psi }(\beta )\right)}\Bigg(6 B_3 \hat{\psi }\left(\alpha +\vartheta ^*\right)^3-6 B_1^3 \left(\alpha +\vartheta ^*\right) \hat{\psi }'\left(\alpha +\vartheta
   ^*\right)^3      \\
   &+6 B_1 \hat{\psi }\left(\alpha +\vartheta ^*\right) \hat{\psi
   }'\left(\alpha +\vartheta ^*\right) \left(B_1^2 \left(\alpha +\vartheta
   ^*\right) \hat{\psi }''\left(\alpha +\vartheta ^*\right)+\hat{\psi
   }'\left(\alpha +\vartheta ^*\right) \left(B_1^2+2 B_2 (\alpha+\vartheta ^*)\right)\right)    \\
   &-6 \hat{\psi}(\alpha +\vartheta^*)^2\hat{\psi }'\left(\alpha +\vartheta
   ^*\right) \left(2 B_1 B_2+  B_3 (\alpha+\vartheta^*)\right)\\
   &-\hat{\psi}(\alpha +\vartheta^*)^2 B_1 \left(3 \hat{\psi}''\left(\alpha +\vartheta ^*\right) \left(B_1^2+2 B_2(\alpha+\vartheta ^*)\right)+ B_1^2 \left(\alpha +\vartheta
   ^*\right) \hat{\psi }^{(3)}\left(\alpha +\vartheta ^*\right)\right)\Bigg)\Bigg].
\end{align*}
\end{proposition}

\begin{theorem}\label{mainresult}
If assumptions {\bf (SN)} hold for a spectrally negative L\'evy process $X$ or {\bf (SP)} hold for a spectrally positive L\'evy process $\hat{X}$,
then the measure $\xi(dx,dy)$ defined formally in (\ref{speeddef}) exists. That is, the speed of convergence to the quasi-stationary distribution is of order $1/t$.
Moreover, the Laplace transform $\tilde{\xi} (\alpha, \beta):=\int_0^\infty \int_0^\infty e^{-\alpha x}e^{-\alpha y}\xi(dx,dy)$ equals:
\begin{equation}
\tilde{\xi} (\alpha, \beta)=C_3(\alpha,\beta)-\tilde{\mu}(\alpha,\beta)C_3(0,0).
\end{equation}
\end{theorem}
\begin{proof}
Recall that under the imposed assumptions {\bf (SN)} or {\bf (SP)}, the Laplace exponent $L({\vartheta};\alpha,\beta)$ as a function of $\vartheta$ satisfies the assumptions of Theorem \ref{doetsch}. Hence, from Proposition \ref{prop:joint.transform-expansion} and Theorem \ref{doetsch} we have that as $t\rightarrow\infty$
\begin{equation}\label{expLTbv}
\EE_\pi[e^{-\alpha Q(0)-\beta Q(t)},T>t]=
e^{\zeta^*}\left(\frac{C_1(\alpha,\beta)}{\Gamma(-1/2)}t^{-3/2}+\frac{C_3(\alpha,\beta)}{\Gamma(-3/2)}t^{-5/2}+o(t^{-5/2})\right).
\end{equation}
Further note that
\begin{equation}\label{QSLT}
\tilde{\mu}(\alpha,\beta)=C_1(\alpha,\beta)/C_1(0,0).
\end{equation}
Now, straightforward calculations give
\[\EE_\pi[e^{-\alpha Q(0)-\beta Q(t)}|T>t]=\frac{C_1(\alpha,\beta)}{C_1(0,0)} +\left(\frac{C_3(\alpha,\beta)}{C_1(0,0)}-\frac{C_1(\alpha,\beta)C_3(0,0)}{C_1^2(0,0)}\right)t^{-1} + o(t^{-1}),
\]
which completes the proof due to the definition of the measure $\xi$.
\end{proof}

\begin{remark}\rm
Formally starting from Proposition \ref{p.joint.transform}, then using generalization of Proposition \ref{prop:joint.transform-expansion} in the next step combined with Theorem \ref{doetsch}, will give more terms of the expansion of $\EE_\pi[e^{-\alpha Q(0)-\beta Q(t)}|T>t]$, and hence a longer expansion of the measure $\PP_\pi(Q(0)\in dx, Q(t)\in dy|T>t)$
as $t\rightarrow\infty$.
\end{remark}

\section{Proof of Proposition \ref{prop:joint.transform-expansion}}\label{mainstep}
Presume now that assumptions {\bf (SP)} hold. We start from deriving the expansion of $\hat{\Phi}(\vartheta)$:
\begin{equation}\label{exp0}
\hat{\Phi}(s)=\vartheta^* + B_1 (s-\zeta^{*})^{1/2} + B_2 ({s}-\zeta^{*}) + B_3 (s-\zeta^{*})^{3/2} + o((s-\zeta^{*})^{3/2})
\end{equation}
as $s\downarrow\zeta^*$. Indeed,
from a Taylor series expansion of $\hat{\psi}$ around $\vartheta^*$ and the condition that $\hat\psi'({\vartheta^*})=0$, we have
\begin{equation}\label{exp1}
\hat{\psi}({\vartheta})-\hat{\psi}({\vartheta^*})= \frac{({\vartheta}-{\vartheta^*})^2}{2} \hat{\psi}''({\vartheta^*})+ \frac{({\vartheta}-{\vartheta^*})^3}{6} \hat{\psi}^{(3)}({\vartheta^*})+ \frac{({\vartheta}-{\vartheta^*})^4}{24} \hat{\psi}^{(4)}({\vartheta^*})+ o(({\vartheta}-{\vartheta^*})^4).
\end{equation}
Rewriting this equation yields
\begin{equation}\label{exp2}
\vartheta-\vartheta^* = B_1 \sqrt{\hat\psi({\vartheta})-\hat\psi({\vartheta^*})}+ B_2\left(\hat\psi({\vartheta})-\hat\psi({\vartheta^*})\right)+B_3 \left(\hat\psi({\vartheta})-\hat\psi({\vartheta^*})\right)^{3/2} + o(({\vartheta}-\vartheta^*)^4).
\end{equation}
The way we derive coefficients $B_i$ ($i=1,2,3$) is by including the expansion (\ref{exp2}) into (\ref{exp1}) and matching respective powers of
$\hat\psi({\vartheta})-\hat\psi({\vartheta^*})$.
Substituting  ${\vartheta}=\hat\Phi(s)$ and using $\hat{\psi}(\hat{\Phi}(s))=s$ completes the proof of (\ref{exp0}).

In the second step, we plug the expansion \eqref{exp0} into Proposition \ref{eq.Lsp.pos} and order the outcome according to powers of $s-\zeta^{*}$.
This will complete the proof of spectrally positive case {\bf (SP)}.

Similarly, when {\bf (SN)} holds then
\[
{\Phi}(s)=\vartheta^* + A_1 (s-\zeta^*)^{1/2} + A_2 ({s}-\zeta^{*}) + A_3 (s-\zeta^{*})^{3/2} + o((s-\zeta^{*})^{3/2})
\]
as $s\downarrow\zeta^*$. Then, using Proposition \ref{p.joint.transform} in the same way as before gives the required assertion after some simple manipulations.
\halmos

\section{Examples}
In this section, we illustrate our theory through a few examples.

\begin{example}[The $M/E(2,\nu)/1$  queue] In this case
\begin{eqnarray}\label{CL}
 X(t) = \sum_{i=1}^{N(t)} \sigma_i-t,
\end{eqnarray}
where $\sigma_i$ (where $i=1,2,\ldots$) are i.i.d.\ service times
that have an Erlang$(2,\nu)$ distribution. The
arrival process is a homogeneous Poisson process $N(t)$ with rate
$\lambda$. We assume that $\varrho:=2\lambda/\nu<1$. For the Laplace exponent, we have that
\[\hat\psi(\eta)=\eta-\lambda+\lambda\left(\frac{\nu}{\eta + \nu}\right)^2\]
which attains its minimum at
$${\vartheta}\s=\sqrt[3]{2\lambda\nu^2} -\nu$$
and it is equal to
$$\zeta\s=\frac{3\sqrt[3]{\lambda\nu^2}}{\sqrt[3]{2}}-\nu-\lambda.$$
One can easy check that all assumptions {\bf (SP)} are satisfied.
In particular, $\hat{\Phi}(z)$ is analytic in
$\mathbb{C}\setminus(-\infty,\zeta{\s}]$.
Then Proposition \ref{prop:joint.transform-expansion} gives  for $C_1(\alpha,\beta)$ that
\begin{multline*}
\frac{2^{5/3} \lambda  (\beta +\nu )^2 (2 \lambda -\nu ) \sqrt{\sqrt[3]{\lambda } \nu ^{2/3}}
   \left(\alpha +\sqrt[3]{2} \sqrt[3]{\lambda } \nu ^{2/3}\right) \left(\alpha +\sqrt[3]{2}
   \sqrt[3]{\lambda } \nu ^{2/3}+2 \nu \right)}{\sqrt{3} \nu  \left(\alpha ^2+2 \sqrt[3]{2}
   \alpha  \sqrt[3]{\lambda } \nu ^{2/3}-\lambda  (\alpha +\nu )+2^{2/3} \lambda ^{2/3} \nu^{4/3}-\sqrt[3]{2} \lambda ^{4/3} \nu ^{2/3}\right)^2}\ \cdot
\\
  \frac{1}{\left(2 \beta ^3-3 \sqrt[3]{2} \beta ^2
   \sqrt[3]{\lambda } \nu^{2/3}+6 \beta ^2 \nu -6 \sqrt[3]{2} \beta  \sqrt[3]{\lambda } \nu^{5/3}+2 \nu ^2 (3 \beta +\lambda )-3 \sqrt[3]{2} \sqrt[3]{\lambda } \nu^{8/3}+2 \nu^3\right)}
\end{multline*}
which is sufficient to compute the bivariate Laplace transform of the quasi-stationary measure given in \eqref{QSLT}. As a numerical illustration, if $\lambda =1$ and $\nu =3$ then
\begin{align*}
&\tilde{\mu}(\alpha,\beta)=\frac{C_1(\alpha,\beta)}{C_1(0,0)}=
-\Bigg\{\left(8-2^{1/3} 3^{5/3}\right)^2 \left(\alpha +2^{1/3} 3^{2/3}\right) \left(\alpha +2^{1/3} 3^{2/3}+6\right) (\beta +3)^2\Bigg\}
\\&:\Bigg\{2 \left(-\alpha ^2-2^{4/3} 3^{2/3} \alpha +\alpha -2^{2/3} 3^{4/3}+2^{1/3}3^{2/3}+3\right)^2 \bigg(-2 \beta ^3+3 \left(2^{1/3} 3^{2/3}-6\right) \beta ^2
\\&+18 \left(2^{1/3} 3^{2/3}-3\right) \beta +9 \left(2^{1/3} 3^{5/3}-8\right)\bigg)\Bigg\}.
\end{align*}

To identify the bivariate Laplace transform of the second order quasi-stationary measure $\xi$ given in Theorem  \ref{mainresult} we need to find $C_3(\alpha, \beta)$.
Unfortunately, its expression is rather complex:
\begin{align*}
&C_3(\alpha,\beta)=\\
&\left(1-\frac{2 \lambda }{\nu }\right) \Bigg\{\bigg(\frac{2^{2/3} \sqrt{\sqrt[3]{\lambda } \nu ^{2/3}} \left(\sqrt[3]{2} \nu ^{2/3} \sqrt[3]{\lambda }+\frac{\nu ^2 \lambda }{\left(\sqrt[3]{2} \nu ^{2/3} \sqrt[3]{\lambda }+\alpha \right)^2}-\lambda +\alpha -\nu \right)}{\sqrt{3}}\\
&-\frac{2^{2/3} \alpha  \sqrt{\sqrt[3]{\lambda } \nu ^{2/3}} \left(\alpha +\sqrt[3]{2}
   \sqrt[3]{\lambda } \nu ^{2/3}-\nu \right) \left(\alpha ^2+3\ \sqrt[3]{2} \alpha
   \sqrt[3]{\lambda } \nu ^{2/3}+3\ 2^{2/3} \lambda ^{2/3} \nu ^{4/3}\right)}{\sqrt{3}
   \left(\alpha +\sqrt[3]{2} \sqrt[3]{\lambda } \nu ^{2/3}\right)^3}\bigg)\\
&/\bigg(\bigg(\sqrt[3]{2} \nu ^{2/3} \sqrt[3]{\lambda }+\frac{\nu ^2 \lambda }{\left(\sqrt[3]{2} \nu ^{2/3} \sqrt[3]{\lambda }+\alpha \right)^2}-\lambda +\alpha -\nu \bigg)^2 \bigg(\frac{3\nu ^{2/3} \sqrt[3]{\lambda }}{2^{2/3}}-\frac{\nu ^2 \lambda }{(\beta +\nu )^2}-\beta -\nu \bigg)^2\bigg)\\
&
-\bigg\{\frac{8 \left(\sqrt[3]{2} \nu ^{2/3} \sqrt[3]{\lambda }+\frac{\nu ^2 \lambda }{\left(\sqrt[3]{2} \nu ^{2/3} \sqrt[3]{\lambda }+\alpha \right)^2}-\lambda +\alpha -\nu \right) \left(1-\frac{2 \lambda  \nu ^2}{\left(\sqrt[3]{2} \nu ^{2/3} \sqrt[3]{\lambda }+\alpha \right)^3}\right)^2 \left(\sqrt[3]{\lambda } \nu ^{2/3}\right)^{3/2}}{\sqrt{3}}\\
&-\frac{8 \sqrt{3} \lambda  \nu ^2 \left(\sqrt[3]{2} \nu ^{2/3} \sqrt[3]{\lambda }+\frac{\nu ^2 \lambda }{\left(\sqrt[3]{2} \nu ^{2/3} \sqrt[3]{\lambda }+\alpha \right)^2}-\lambda +\alpha -\nu \right)^2 \left(\sqrt[3]{\lambda } \nu ^{2/3}\right)^{3/2}}{\left(\sqrt[3]{2} \nu ^{2/3} \sqrt[3]{\lambda }+\alpha \right)^4}\\
&-\frac{8 \alpha ^3 \left(\sqrt[3]{\lambda } \nu ^{2/3}\right)^{3/2} \left(\alpha +\sqrt[3]{2}
   \sqrt[3]{\lambda } \nu ^{2/3}-\nu \right) \left(\alpha ^2+3 \sqrt[3]{2} \alpha
   \sqrt[3]{\lambda } \nu ^{2/3}+3\sqrt[3]{4} \lambda ^{2/3} \nu ^{4/3}\right)^3}{\sqrt{3}
   \left(\alpha +\sqrt[3]{2} \sqrt[3]{\lambda } \nu ^{2/3}\right)^9}\\
&+\frac{16 \sqrt{3} \alpha  \lambda  \nu ^2 \left(\sqrt[3]{2} \nu ^{2/3} \sqrt[3]{\lambda }+\frac{\nu ^2 \lambda }{\left(\sqrt[3]{2} \nu ^{2/3} \sqrt[3]{\lambda }+\alpha \right)^2}-\lambda +\alpha -\nu \right) \left(1-\frac{2 \lambda  \nu ^2}{\left(\sqrt[3]{2} \nu ^{2/3} \sqrt[3]{\lambda }+\alpha \right)^3}\right) \left(\sqrt[3]{\lambda } \nu ^{2/3}\right)^{3/2}}{\left(\sqrt[3]{2} \nu ^{2/3} \sqrt[3]{\lambda }+\alpha \right)^4}\\
&+\frac{16 \sqrt{3} \lambda  \left(\sqrt[3]{2} \sqrt[3]{\lambda } \nu ^{2/3}-\nu \right) \nu ^2 \left(\sqrt[3]{2} \nu ^{2/3} \sqrt[3]{\lambda }+\frac{\nu ^2 \lambda }{\left(\sqrt[3]{2} \nu ^{2/3} \sqrt[3]{\lambda }+\alpha \right)^2}-\lambda +\alpha -\nu \right) \left(1-\frac{2 \lambda  \nu ^2}{\left(\sqrt[3]{2} \nu ^{2/3} \sqrt[3]{\lambda }+\alpha \right)^3}\right) \left(\sqrt[3]{\lambda } \nu ^{2/3}\right)^{3/2}}{\left(\sqrt[3]{2} \nu ^{2/3} \sqrt[3]{\lambda }+\alpha \right)^4}\\
&+\frac{32 \lambda  \nu ^2 \left(\sqrt[3]{\lambda } \nu ^{2/3}\right)^{3/2} \left(\alpha +\sqrt[3]{2} \sqrt[3]{\lambda } \nu ^{2/3}-\nu \right)^3 \left(\alpha ^2+2 \sqrt[3]{2} \alpha  \sqrt[3]{\lambda } \nu
   ^{2/3}-\alpha  \lambda +2^{2/3} \lambda ^{2/3} \nu ^{4/3}-\sqrt[3]{2} \lambda ^{4/3} \nu ^{2/3}-\lambda  \nu \right)^2}{\sqrt{3} \left(\alpha +\sqrt[3]{2} \sqrt[3]{\lambda } \nu ^{2/3}\right)^9}\\
\end{align*}
\begin{align*}
&+\frac{16\sqrt[3]{4} \alpha \lambda   \nu ^2 \left(\sqrt[3]{2} \nu ^{2/3} \sqrt[3]{\lambda }+\frac{\nu ^2 \lambda }{\left(\sqrt[3]{2} \nu ^{2/3} \sqrt[3]{\lambda }+\alpha \right)^2}-\lambda +\alpha -\nu \right)^2 \sqrt{\sqrt[3]{\lambda } \nu ^{2/3}}}{\sqrt{3} \left(\sqrt[3]{2} \nu ^{2/3} \sqrt[3]{\lambda }+\alpha \right)^4}\\
&+\frac{16\sqrt[3]{4}  \lambda  \left(\sqrt[3]{2} \sqrt[3]{\lambda } \nu ^{2/3}-\nu \right) \nu ^2 \left(\sqrt[3]{2} \nu ^{2/3} \sqrt[3]{\lambda }+\frac{\nu ^2 \lambda }{\left(\sqrt[3]{2} \nu ^{2/3} \sqrt[3]{\lambda }+\alpha \right)^2}-\lambda +\alpha -\nu \right)^2 \sqrt{\sqrt[3]{\lambda } \nu ^{2/3}}}{\sqrt{3} \left(\sqrt[3]{2} \nu ^{2/3} \sqrt[3]{\lambda }+\alpha \right)^4}\\
&-\left(\alpha ^2+3 \sqrt[3]{2} \alpha
   \sqrt[3]{\lambda } \nu ^{2/3}+3\sqrt[3]{4} \lambda ^{2/3} \nu ^{4/3}\right) \left(\alpha ^2+2
   \sqrt[3]{2} \alpha  \sqrt[3]{\lambda } \nu ^{2/3}-\lambda  (\alpha +\nu )+2^{2/3} \lambda
   ^{2/3} \nu ^{4/3}-\sqrt[3]{2} \lambda ^{4/3} \nu ^{2/3}\right)\cdot\\
    &\qquad\frac{16\sqrt[3]{4} \alpha  \lambda  \sqrt{\sqrt[3]{\lambda } \nu ^{2/3}} \left(\alpha
   +\sqrt[3]{2} \sqrt[3]{\lambda } \nu ^{2/3}-\nu \right)^3 \left(\alpha +\sqrt[3]{2}
   \sqrt[3]{\lambda } \nu ^{2/3}+2 \nu \right) }{3 \sqrt{3} \left(\alpha
   +\sqrt[3]{2} \sqrt[3]{\lambda } \nu ^{2/3}\right)^8}\\
&+86 \sqrt[3]{2} \lambda  \left(\alpha +\sqrt[3]{2} \sqrt[3]{\lambda } \nu ^{2/3}-\nu
   \right)^4\cdot\\
   &\qquad\frac{ \left(\alpha +\sqrt[3]{2} \sqrt[3]{\lambda } \nu ^{2/3}+2 \nu \right) \left(\alpha
   ^2+2 \sqrt[3]{2} \alpha  \sqrt[3]{\lambda } \nu ^{2/3}-\alpha  \lambda +2^{2/3} \lambda
   ^{2/3} \nu ^{4/3}-\sqrt[3]{2} \lambda ^{4/3} \nu ^{2/3}-\lambda  \nu \right)^2}{9 \sqrt{3}
   \sqrt{\sqrt[3]{\lambda } \nu ^{2/3}} \left(\alpha +\sqrt[3]{2} \sqrt[3]{\lambda } \nu
   ^{2/3}\right)^7}\bigg\}\\
&/\Bigg\{6 \left(\sqrt[3]{2} \nu ^{2/3} \sqrt[3]{\lambda }+\frac{\nu ^2 \lambda }{\left(\sqrt[3]{2} \nu ^{2/3} \sqrt[3]{\lambda }+\alpha \right)^2}-\lambda +\alpha -\nu \right)^4 \left(\frac{3\nu ^{2/3} \sqrt[3]{\lambda }}{2^{2/3}}-\frac{\nu ^2 \lambda }{(\beta +\nu )^2}-\beta -\nu \right)\Bigg\}\Bigg\}.
\end{align*}
The expression is easy to evaluate numerically. Taking $\lambda =1$ and $\nu =3$ gives
\begin{align*}
\tilde{\xi}(\alpha,\beta)=\frac{I_1(\alpha, \beta)}{I_2(\alpha, \beta},
\end{align*}
where
\begin{align*}
&I_1(\alpha,\beta)=
2^{5/3} \left(\alpha +2^{1/3} 3^{2/3}\right) \left(\alpha +2^{1/3}3^{2/3}+6\right) (\beta +3)^2\\
&I_2(\alpha,\beta)=
3^{7/6}\left(-\alpha ^2-2^{4/3} 3^{2/3} \alpha +\alpha -3^{4/3}\ 2^{2/3} +2^{1/3} 3^{2/3}+3\right)^2 \left(-2 \beta ^3+
3 \left(2^{1/3}3^{2/3}-6\right) \beta ^2\right.\\
&\left.\hspace{2cm}+18 \left(2^{1/3} 3^{2/3}-3\right) \beta +9 \left(3\sqrt[3]{2} 3^{2/3}-8\right)\right).
\end{align*}
Observe that Laplace transforms $\tilde{\mu}(\alpha,\beta)$ and $\tilde{\xi}(\alpha,\beta)$ can be inverted as they are linear combinations
of powers of $1/({\rm const}_\alpha+\alpha)$ and $1/({\rm const}_\beta+\beta)$
for various constants ${\rm const}_\alpha$ and ${\rm const}_\beta$.
\end{example}

\begin{example}[Linear Brownian motion] In this case $X(t)=
B(t)-t,$ where $B(t)$ is a standard Brownian
motion. Note that this process is spectrally positive {\it and}
spectrally negative. We will apply the spectrally positive results.
It is not hard to check that
\[\hat{\psi}({\vartheta})={\vartheta}+\frac{{\vartheta}^2}{2},\]
so that $\vartheta\s=
-1$ and $\zeta\s=-1/2.$
Furthermore,
$$\widehat{\Phi}(s)=-\left(1+\sqrt{1+2s}\right)
$$
and assumptions {\bf (SP)} are satisfied.

It is a matter of
straightforward computations now to obtain that
\[C_1(\alpha,\beta)=-\frac{4 \sqrt{2}}{(\alpha +1)^2 (\beta +1)^2}\]
and thus
\[\tilde{\mu}(\alpha,\beta)=\frac{C_1(\alpha,\beta)}{C_1(0,0)}=\left(\frac{1}{1 +\alpha}\right)^2\left(
\frac{1}{1+\beta}\right)^2.\]
Thus, the quasi-stationary
distributions of $Q(0)$ and $Q(t)$ (conditioned that busy period lasts longer than $t$, for large $t$) are both Erlang(2)
with mean $2$, whereas the stationary workload itself has an
exponential distribution with mean $1/2$; see
\cite{Iglehart,Martmart,MPR}.
Moreover, simple calculations lead to:
\[
C_3(\alpha,\beta)=-\frac{8 \sqrt{2} (2 + \alpha (2 + \alpha) + \beta (2 + \beta))}{
(1 + \alpha)^4 (1 + \beta)^4}.
\]
Hence,
\[\tilde{\xi}(\alpha,\beta)=
\frac{2}{(\alpha+1)^4 (\beta+1)^2}+\frac{2}{(\alpha+1)^2 (\beta+1)^4}-\frac{4}{(\alpha+1)^2 (\beta+1)^2}.
\]
One can observe that the {\it second order quasi-stationary measure} $\xi$ is given by
\[\xi(dx,dy)/dxdy=2f_{4,2}(x,y)+2f_{2,4}(x,y)-4f_{2,2}(x,y),
\]
where $f_{i,j}(x,y)$ is a density of the bivariate distribution of two independent Erlang distributions ${\rm Erlang}(1,i)$ and ${\rm Erlang}(1,j)$.
\end{example}


\end{document}